\documentclass{article}

\usepackage[english]{babel}

\usepackage[letterpaper,top=2cm,bottom=2cm,left=3cm,right=3cm,marginparwidth=1.75cm]{geometry}

\usepackage{authblk}

\usepackage{enumitem}

\usepackage{graphicx}
\usepackage{color}
\usepackage{tikz, tikz-3dplot}
\usepackage{caption}
\usetikzlibrary{positioning,math,quotes,angles}

\usepackage{amsmath,amsfonts,amsthm,bm}
\numberwithin{equation}{section}

\theoremstyle{plain}
\newtheorem{dfn}{Definition}[section]
\newtheorem*{dfn*}{Definition}
\newtheorem{thm}[dfn]{Theorem}
\newtheorem*{thm*}{Theorem}
\newtheorem{lem}[dfn]{Lemma}
\newtheorem*{lem*}{Lemma}

\newtheorem*{cor*}{Corollary}
\newtheorem{prop}[dfn]{Proposition}
\newtheorem*{prop*}{Proposition}
\newtheorem{rem}[dfn]{Remark}
\newtheorem*{rem*}{Remark}

\title{A Minimax Approach to Relative Periodic Orbits in Symmetric Three-Degree-of-Freedom Hamiltonian Systems}
\author[1]{Shu Sakaguchi}
\author[1]{Mitsuru Shibayama}
\affil[1]{
Department of Applied Mathematics and Physics,
Graduate School of Informatics, Kyoto University,
Yoshida-honmachi, Sakyo-ku, Kyoto, 606-8501, Japan}
\begin{document}
\date{}
\maketitle

\begin{abstract}
We study three-degree-of-freedom Hamiltonian systems that are invariant under rotations about the $z$-axis and under reflection across the $xy$-plane. 
Fixing the angular momentum, such systems reduce to Hamiltonian systems with two degrees of freedom. 
We focus on the range of energy values for which the corresponding Hill regions are compact.
First, under suitable assumptions on the topology of these compact Hill regions, we prove the existence of periodic solutions on each prescribed energy surface of the reduced system by means of a variational minimax method. 
These periodic solutions are obtained as saddle points of the Maupertuis functional. 
The resulting solutions are either nontrivial spatial periodic solutions or trivial planar brake solutions in the reduced system.
Next, by computing the Morse index, we provide a sufficient condition ensuring that the periodic solutions obtained are nontrivial. 
Finally, we apply our results to the isosceles three-body problem and to the spatial anisotropic Kepler problem. 
In both cases, we verify the sufficient condition for nontriviality and thereby establish the existence of nontrivial periodic solutions.
\end{abstract}

\section{Introduction}

Periodic and quasi-periodic solutions to the Newtonian $N$-body problem have been studied for a long time.
Using variational methods, Chenciner and Montgomery \cite{ChencinerMontgomery00} constructed the celebrated figure-eight periodic solution of the planar three-body problem.
Since then, a number of periodic and quasi-periodic solutions have been found as minimizers of the action functional.

On the other hand, the prescribed-energy problem, which seeks periodic solutions on a fixed energy level, has also been widely studied.
In this setting, solutions are characterized as critical points of the Jacobi-Maupertuis metric or of the Maupertuis functional.
When the Hill region is compact, the existence of brake orbits has been established by Benci \cite{Benci84}, Benci and Giannoni \cite{BenciGiannoni89}, Bolotin \cite{Bolotin78}, Gluck and Ziller \cite{GluckZiller83}, and Hayashi \cite{Hayashi83}. 
When the Hill region contains singularities, as in the Kepler problem, it becomes more difficult to prove that solutions are collision-free. Some results establish the existence of generalized periodic solutions, which may include collisions (see, for instance, \cite{AmbrosettiCotiZelati93}).

In this paper, we use the Maupertuis functional to prove the existence of collision-free periodic solutions with prescribed energy and angular momentum. 
The angular momentum constraint makes the Hill region compact, which allows us to avoid collisions.

We consider the following three-degree-of-freedom Hamiltonian system:
\begin{align*}
    \frac{d^2q}{dt^2} &= - \nabla V_0(q),\qquad q=(x,y,z) \in \mathbb{R}^3,
\end{align*}
where the potential function $V_0$ depends on $x$ and $y$ only through $\sqrt{x^2 + y^2}$.
The Hamiltonian is given by
\begin{equation*}
    H(x,y,z,p_x,p_y,p_z) = \frac12 (p_x^2 +p_y^2 +p_z^2) + V_0(x,y,z).
\end{equation*}
In cylindrical coordinates $(r,\theta,z)$, the Hamiltonian becomes
\begin{equation*}
    H(r,\theta,z,p_r,p_\theta,p_z) = \frac12 \left( p_r^2 + \frac{p_\theta^2}{r^2} + p_z^2 \right) + V_0(r,z).
\end{equation*}

Now, since $\theta$ is cyclic, the angular momentum $p_\theta$ is conserved.
We fix the angular momentum by setting $p_\theta = \omega>0$ and obtain the reduced Hamiltonian
\begin{equation}\label{eq:hamiltonian}
    \tilde{H}(r,z,p_r,p_z) = \frac12 (p_r^2 + p_z^2) + \frac{\omega^2}{2r^2} + V_0(r,z).
\end{equation}
We denote the effective potential $\omega^2/(2r^2) + V_0(r,z)$ by $V(r,z)$.

We assume that the effective potential $V$ satisfies the following properties:
\begin{enumerate}[label=(V\arabic*)]
    \item $V(r,z)$ is $C^2$;
    \item $V(r,z) = V(r,-z)$;
    \item There exist $h_- < h_+$ such that, for every $h \in (h_-,h_+)$, the Hill region $\Omega_h := \{ (r,z) \mid V(r,z) \leq h \}$ is compact and simply connected, does not contain the origin, and satisfies
    \begin{equation*}
        \nabla V(r,z) \neq 0 \quad \text{on } \partial \Omega_h
        = \{ (r,z) \mid V(r,z) = h \}.
    \end{equation*}
    Moreover, there exists $\varepsilon > 0$ such that, for
    $J = \{ r \in \mathbb{R}_+ \mid (r,0) \in \operatorname{int} \Omega_{h+2\varepsilon} \}$, there exists a continuous function $\zeta:J \to (0,\infty)$ such that
    \begin{equation*}
        \partial \Omega_{h+2\varepsilon} \cap \{(r,z) \mid z>0\}
        =
        \{(r,\zeta(r)) \mid r \in J\}.
    \end{equation*}
\end{enumerate}

\begin{rem}
    If the Hill region $\Omega_h$ is not connected and a connected component of $\Omega_h$ satisfies (V3), then we may take this connected component as the Hill region.
\end{rem}

Although assumption (V3) may seem complicated, the condition involving $\zeta$ is satisfied if
\begin{equation*}
    \partial_z V(r,z) > 0
    \quad \text{in } \Omega_{h+2\varepsilon} \cap \{(r,z) \mid z>0\}.
\end{equation*}
This condition holds in our two applications: the isosceles three-body problem and the spatial anisotropic Kepler problem.\\

We state our main results.

\begin{thm}\label{thm:main}
    Consider the Hamiltonian system with Hamiltonian \eqref{eq:hamiltonian}.
    If the potential function $V(r,z)$ satisfies assumptions (V1)--(V3), then for every $h \in (h_-, h_+)$, there exist periodic solutions with energy $h$ such that
    \begin{equation}\label{eq:sol-symmetry}
        z(0) = z(T/2) =0, \quad \dot r(0) = \dot r(T/2) = 0,
    \end{equation}
    where $T$ is the period.
\end{thm}

The periodic orbits obtained for the reduced system correspond to relative periodic orbits of the original system with energy $h$ and angular momentum $\omega$.
Let
\begin{equation*}
    \Delta \theta
    =
    \int_0^T \frac{\omega}{r(t)^2}\,dt
\end{equation*}
be the change in the angular variable over one period of the reduced solution.
If $\Delta\theta/(2\pi)$ is rational, then the corresponding solution in the original system is periodic; if $\Delta\theta/(2\pi)$ is irrational, then it is quasi-periodic. 

The equation \eqref{eq:sol-symmetry} gives rise to two possibilities: spatial periodic orbits that meet the $r$-axis orthogonally at $t=0$ and $t=T/2$, and planar brake orbits.
Here, a brake orbit means a periodic orbit that passes through a point with zero velocity and reverses its motion under time reversal.
The existence of the brake orbits in the reduced system is trivial because of assumption (V2) and (V3).
We have to distinguish the obtained solutions from the trivial solutions, which is completed by the computation of the Morse index in each application.
\\

This paper is organized as follows.
In Section 2, we prove Theorem \ref{thm:main} using a variational minimax method based on the Maupertuis functional.
We also show that the critical points obtained in this way have Morse index at most $2$.
In Section 3, we study the Morse indices of planar solutions and provide a sufficient condition under which the Morse index of a planar solution is greater than 2.
In Section 4, we apply our results to the isosceles three-body problem.
In Section 5, we apply our results to the spatial anisotropic Kepler problem.

\section{Variational methods}

In this section, we prove Theorem \ref{thm:main} using minimax methods.
We use the Maupertuis functional with modifications employed by Benci and Giannoni \cite{BenciGiannoni89}. 
The proofs of some lemmas and the main theorem are based on those of \cite{BenciGiannoni89} with some modification.

\subsection{Preliminaries}

To prove the main theorem, we use a variant of linking theorem (see \cite{Rabinowitz86}).

\begin{thm}\label{thm:linking}
    Let $\Lambda$ be an open subset of a real Hilbert space $X$ and $0 \in \Lambda$.
    We consider $I \in C^1(\Lambda, \mathbb{R})$ with $I(0) \leq 0$.
    We assume the following:
    \begin{enumerate}[label=(I\arabic*)]
        \item If $\gamma_k \to \gamma \in \partial \Lambda\ ( k \to \infty)$, then $\limsup_{k \to \infty} I(\gamma_k) \leq 0$;
        \item There exists an $N$-dimensional subspace $E_N$ with $N \geq 1$ such that
        \begin{enumerate}[label=(\roman*)]
            \item $I|_{E_N \cap \Lambda} \leq 0$;
            \item There exists $\rho > 0,\ \alpha > 0$ such that 
            \begin{align*}
                &\bar B_\rho := \{\gamma \in X \mid \|\gamma\| \leq \rho\} \subset \Lambda \\
                & \inf_{S} I \geq \alpha
            \end{align*}
            where $S := \partial \bar B_\rho \cap E_N^\perp,\ E_N^\perp := \{v \in X \mid \langle v, w \rangle = 0 \ \text{for all }  w \in E_N\}$;
            \item There exists $ e \in E_N^\perp \setminus \{0\}$ such that the set
            \begin{equation*}
                Q_\Lambda := \{ y + se \mid y \in E_N,\ s \geq 0\} \cap \Lambda
            \end{equation*}
            is bounded, diffeomorphic to $(B_1 \cap E_N) \times [0,1)$, and $\sup_{Q_\Lambda} I =: \beta < +\infty$ where $B_1$ denotes open unit ball;
        \end{enumerate}
        \item There exists $\varepsilon_0 > 0$ such that $I|_\Lambda$ satisfies Palais-Smale condition (PS) on the interval $[\alpha - \varepsilon_0, \beta + \varepsilon_0]$, i.e., if $(\gamma_k) \subset \Lambda$ satisfies
        \begin{align*}
            &I(\gamma_k) \to c \in [\alpha - \varepsilon_0, \beta+\varepsilon_0] \qquad \text{as} \qquad k \to \infty,\\
            &\|I'(\gamma_k)\|_{X\!^*} \to 0 \qquad \text{as} \qquad k \to \infty, 
        \end{align*}
        then $(\gamma_k)$ possesses a strongly convergent subsequence.
    \end{enumerate}
    Then, $I$ has a critical point $\gamma \in \Lambda$ satisfying
    \begin{equation*}
        \alpha \leq I(\gamma) \leq \beta.
    \end{equation*}
\end{thm}

For the proof of this theorem, we use the deformation lemma.
For $a \in \mathbb{R}$, we write $[I \leq a] := \{\gamma \in \Lambda \mid I(\gamma) \leq a\}$. 

\begin{lem}[Deformation lemma]\label{lem:deformation}
    Let $a<b$ and $\Lambda$ be an open subset of a real Hilbert space $X$.
    Assume $I \in C^1(\Lambda, \mathbb{R})$ has no critical values on $[a,b]$ and satisfies (PS) on $[a,b]$.
    Moreover, assume if $\gamma_k \to \gamma \in \partial \Lambda$, then $\limsup I(\gamma_k) < a$.
    Then there exists $\eta \in C([0,1] \times \Lambda, \Lambda)$ such that
    \begin{itemize}
        \item $\eta(0, \gamma)=\gamma \quad \forall \gamma \in [I \leq b]$;
        \item $\eta(t,\gamma) = \gamma \quad \forall (t, \gamma) \in [0,1] \times [I\leq a]$;
        \item For any $\gamma \in [I \leq b]$, the map $t \mapsto I(\eta(t,\gamma))$ is nonincreasing;
        \item For any $\gamma \in [I \leq b]$, $\eta(1,\gamma) \in [I \leq a]$.
    \end{itemize}
\end{lem}

Although the lemma differs from the standard deformation lemma slightly in that the functional $I$ is defined not on the whole Hilbert space but on the subset $\Lambda$, the proof remains essentially the same. 
Indeed, because the values of the functional on the boundary $\partial \Lambda$ do not affect the level $[a,b]$, we can construct the negative gradient flow without any difficulty. 

The proof of the standard deformation lemma can be found in \cite{Rabinowitz86}.

\begin{proof}[Proof of Theorem \ref{thm:linking}]
    Set $F := E_N \oplus \mathbb{R}_{\geq}e$.
    By (I2) (i) and (iii), there exists a closed subset $Q \subset F \cap \Lambda$ satisfying the following:
    \begin{itemize}
        \item $Q \subset Q_\Lambda \subset F$;
        \item $Q$ is homeomorphic to $(\bar{B}_1 \cap E_N) \times [0,1]$;
        \item $\rho e \in Q$;
        \item $I |_{\partial Q} \leq \alpha/2$.
    \end{itemize}
    Indeed, set
    \begin{equation*}
        H := \{y+se\in F \mid \|y\|<1,\ 0\le s<1\}.
    \end{equation*}
    By (I2)(iii), there exists a homeomorphism $h:H\to Q_\Lambda$.
    Moreover, since $\Lambda$ is open, we may assume that
    \begin{equation*}
        h\bigl((B_1\cap E_N)\times\{0\}\bigr)=E_N\cap Q_\Lambda.
    \end{equation*}
    
    Next, set
    \begin{equation*}
        C:=\overline{Q_\Lambda}^{\,F}\setminus Q_\Lambda.
    \end{equation*}
    By (I1), there exists a neighborhood $U\subset F$ of $C$ in $F$ such that
    \begin{equation*}
        I<\alpha/2 \qquad \text{on } U\cap\Lambda.
    \end{equation*}
    
    For $0<r<1$, define
    \begin{equation*}
        K_r:=\{y+se\in F\mid \|y\|\le r,\ 0\le s\le r\}.
    \end{equation*}
    If $r$ is chosen sufficiently close to $1$, then
    \begin{equation*}
        h\bigl(\partial K_r\setminus((B_r\cap E_N)\times\{0\})\bigr)\subset U.
    \end{equation*}
    Since $h^{-1}(\rho e)\in H$, by taking $r$ even closer to $1$ if necessary,
    we may also assume that $\rho e\in h(K_r)$.
    
    We now set
    \begin{equation*}
        Q:=h(K_r).
    \end{equation*}
    Then $Q\subset Q_\Lambda$, $Q$ is homeomorphic to
    $(\bar{B}_1\cap E_N)\times[0,1]$, $\rho e\in Q$, and
    \begin{equation*}
        \partial Q\subset U\cup E_N.
    \end{equation*}
    Therefore, by (I2)(i) and $I|_{U\cap\Lambda}<\alpha/2$, we obtain
    \begin{equation*}
        I|_{\partial Q}\le \alpha/2.
    \end{equation*}
    
    Define
    \begin{align*}
        &\mathcal{F} := \{ f \in C(Q,\Lambda) \mid f|_{\partial Q} = \mathrm{id}  \},\\
        &c := \inf_{f \in \mathcal{F}} \max_{\gamma \in Q} I(f(\gamma)).
    \end{align*}
    Then we obtain a critical point whose critical value lies in $[\alpha,\beta]$.
    Indeed, for every $f\in\mathcal F$, one has
    \begin{equation*}
        f(Q) \cap S \neq \emptyset,
    \end{equation*}
    which follows from Brouwer's fixed point theorem.
    Hence $c\geq \alpha$.
    On the other hand, since $\mathrm{id}\in\mathcal F$, we have $c\leq \beta$.

    Suppose, by contradiction, that $c$ is not a critical value.
    Take $\varepsilon_0$ as in (I3), and apply Lemma \ref{lem:deformation} with
    \begin{equation*}
        a=c-\delta,\qquad b=c+\delta
    \end{equation*}
    where $\delta$ is chosen so that no critical value includes $[c-\delta,c+\delta]$ with $ \delta < \min \{\varepsilon_0, c-\alpha/2\}$.
    Choose $f_0\in\mathcal F$ such that
    \begin{equation*}
        \max_{\gamma \in Q} I(f_0(\gamma)) < c + \delta/2.
    \end{equation*}
    Since $I|_{\partial Q} \leq \alpha/2 < c - \delta$, we have
    \begin{equation*}
        \eta(1,f_0(\cdot)) \in \mathcal{F}.
    \end{equation*}
    Moreover,
    \begin{equation*}
        \max_{\gamma \in Q} I(\eta(1,f_0(\gamma))) \leq c - \delta  < c,
    \end{equation*}
    which contradicts the definition of $c$.
    Therefore, $c$ is a critical value, and the proof is complete.
\end{proof}

\subsection{Setting of the variational problem}
In this subsection, we describe the setting of the variational problem.
We use the Maupertuis functional with a slight modification.
We also show that, even with this modification, critical points yield solutions to the original problem.

Let $h \in (h_-,h_+)$ and $\Omega=\Omega_h$.
The original Maupertuis functional is
\begin{equation*}
    I(r,z,r', z') = \frac 12\int_0^1 (r'^2 + z'^2) \,dt \int_0^1 ( h - V(r,z))\,dt
\end{equation*}
considered on the class of curves
\begin{align*}
    &\Gamma = \{ \gamma(t)= (r(t),z(t)) \in  X \mid \gamma([0,1]) \subset \Omega \},\quad
    \Omega = \Omega_h =  \{(r,z) \mid V(r,z) \leq h\},\\
    &X = H^1([0,1], \mathbb{R}) \times H_0^1([0,1], \mathbb{R})
\end{align*}
where $H^1([0,1],\mathbb{R})$ is the Sobolev space and $H_0^1([0,1],\mathbb{R})$ is defined as the closure, in the $H^1([0,1],\mathbb{R})$-norm, of the space $C_0^\infty((0,1),\mathbb{R})$ of smooth functions with compact support in $(0,1)$. Equivalently, it consists of $H^1$-functions on $[0,1]$ whose boundary values vanish at 0 and 1.

We endow $X$ with the inner product
\begin{equation*}
    \langle \gamma_1, \gamma_2 \rangle_X
    = \langle (r_1, z_1), (r_2, z_2) \rangle_X
    := \int_0^1  (r_1 r_2 + r_1' r_2') \,dt
    + \int_0^1 (z_1z_2 + z_1' z_2') \,dt.
\end{equation*}

We now introduce the following modification.

Take $\varepsilon>0$ as in (V3).
After replacing $\varepsilon>0$ with a smaller one if necessary, we may assume that
\begin{equation*}
    \nabla V(r,z) \neq 0 
\end{equation*}
whenever $V(r,z) \in [h, h + 2\varepsilon]$.
Now choose a $C^\infty$ function $\varphi:(-\infty, h + 2 \varepsilon) \to \mathbb{R}$ such that
\begin{align*}
    \varphi(\tau) &= \begin{cases}
        \tau & \tau \in (-\infty, h + \varepsilon],\\
        \dfrac{1}{(\tau - h - 2\varepsilon)^2} & \tau \in [h + \frac{3 \varepsilon}{2}  , h + 2 \varepsilon),
    \end{cases}\\
    \varphi'(\tau)  &\geq 1,
\end{align*}
where if necessary we may replace $\varepsilon$ with a smaller one again such that $\varphi(h+\frac{3\varepsilon}{2})>h+\frac{3\varepsilon}{2}$.
Define
\begin{equation*}
     \hat{V}(r,z) = \varphi(V(r,z)).
\end{equation*}
We then consider the new functional
\begin{equation*}
    \hat{I}(\gamma)
    = \frac 12 \int_0^1 |\gamma'(t)|^2\, dt
    \int_0^1 (h - \hat{V}(\gamma(t)))\, dt
    = \frac 12 \int_0^1(r'(t)^2 + z'(t)^2)\,dt
    \int_0^1 (h - \hat V(r(t),z(t)))\,dt
\end{equation*}
on the class of curves
\begin{equation*}
    \hat{\Gamma}
    = \{ \gamma(t)= (r(t),z(t)) \in X\ |\ \gamma([0,1]) \subset \hat \Omega \}
\end{equation*}
where
\begin{equation*}
    \hat{\Omega} := \mathrm{int}\ \Omega_{h+2\varepsilon}
    = \{(r,z) \mid V(r,z) < h + 2\varepsilon\}.
\end{equation*}
The restriction of $\hat{I}$ to $\Gamma$ coincides with $I$.
In what follows, we write
\begin{equation*}
    A(\gamma) = \int_0^1 |\gamma'(t)|^2\, dt,\quad
    B(\gamma) =  \int_0^1 (h - \hat{V}(\gamma(t)))\, dt.
\end{equation*}

\begin{lem}\label{lem:critical_pt}
    Suppose that $\gamma_\ast = (r_\ast ,z_\ast) \in \hat{\Gamma}$ is a critical point of the functional $\hat{I}$ with $\hat I(\gamma_\ast) > 0$.
    Then $\gamma_\ast([0,1]) \subset \bar{\Omega}$ and $r_\ast'(0) = r_\ast'(1) = 0$.
    Moreover, by setting the curve equal to $(r_\ast(2-t), -z_\ast(2-t))$ for $t \in [1,2]$, it can be extended smoothly.
    Furthermore, if
    \begin{equation*}
        \lambda = \frac{A(\gamma_\ast)}{2B(\gamma_\ast)}
        = \frac{ \int_0^1 |\gamma_\ast'|^2 dt}
        {2\int_0^1 (h - \hat V(\gamma_\ast))dt},
    \end{equation*}
    then $\lambda > 0$, and the curve
    $(r_\ast(s/\sqrt{\lambda}), z_\ast (s/ \sqrt{\lambda}))$
    is a $2\sqrt{\lambda}$-periodic solution with energy $h$.
\end{lem}

\begin{proof}
    By the assumption $\hat{I}(\gamma_\ast) > 0$, we have
    $A(\gamma_\ast), B(\gamma_\ast)>0$.
    
    Since $\gamma_\ast = (r_\ast , z_\ast)$ is a critical point of $\hat{I}$, for any $\gamma = (r,z) \in X$ we have
    \begin{align*}
         \int_0^1 \langle \gamma_\ast', \gamma' \rangle dt
         \int_0^1 (h - \hat V (\gamma_\ast))dt
         -  \frac 12\int_0^1 |\gamma_\ast'|^2dt
         \int_0^1 \langle \nabla \hat V(\gamma_\ast), \gamma \rangle dt = 0
    \end{align*}
    where $\langle\cdot, \cdot \rangle$ denotes the standard Euclidean inner product.
    In particular, by taking $\gamma \in H^1_0\times H_0^1$, we obtain
    \begin{equation} \label{eq:lem1_1}
        B(\gamma_\ast) \gamma_\ast''(t)
        + \frac12 A(\gamma_\ast) \nabla\hat V(\gamma_\ast (t)) = 0
        \quad \forall t \in [0,1].
    \end{equation}
    Initially, this equation holds in the sense of $\mathcal{D}'((0,1), \mathbb{R}^2)$.
    However, since $A(\gamma_\ast), B(\gamma_\ast)>0$, $\hat{V}$ is smooth, and the Sobolev embedding theorem applies, we obtain $\gamma_\ast \in C^2([0,1], \mathbb{R}^2)$.
    Hence the equation holds for every $t \in [0,1]$.
    
    Since $z$ vanishes at the endpoints, integration by parts gives
    \begin{align*}
         \left([ r_\ast' r]_0^1
         - \int_0^1 \langle \gamma_\ast'', \gamma \rangle dt \right)
         \int_0^1 (h - \hat V(\gamma_\ast))dt
         - \frac 12 \int_0^1 |\gamma_\ast'|^2dt
         \int_0^1 \langle \nabla \hat V(\gamma_\ast), \gamma \rangle dt = 0.
    \end{align*}
    Therefore, since $B(\gamma_\ast) > 0$, we also have
    \begin{align} \label{eq:lem1_2}
        r_\ast' (0) = r_\ast'(1) = 0.
    \end{align}
    From \eqref{eq:lem1_1}, it follows that
    \begin{align}\label{eq:lem1_3}
        \frac12 B(\gamma_\ast)  |\gamma_\ast'(t)|^2
        - \frac12 A(\gamma_\ast) (h - \hat V (\gamma_\ast(t)))
        = \text{const}.      
    \end{align}
    Integrating both sides from $0$ to $1$, we see that this constant is equal to $0$.
    Since $A(\gamma_\ast), B(\gamma_\ast) > 0$, \eqref{eq:lem1_3} yields
    \begin{equation*}
        h - \hat V (\gamma_\ast (t))
        = \frac{B(\gamma_\ast)}{A(\gamma_\ast)}|\gamma_\ast'(t)|^2
        \geq 0 \quad \forall t \in [0,1].
    \end{equation*}
    Thus, recalling that $\hat{V}\leq h$ holds only on $\bar \Omega$, we obtain
    $\gamma_\ast([0,1]) \subset \bar{\Omega}$.
    
    The functional $\hat{I}$ is invariant under the transformations
    $t \mapsto 1-t$ and $z \mapsto -z$.
    Hence $(r_\ast(1-t), -z_\ast(1-t))$ is also a critical point.
    Moreover, by \eqref{eq:lem1_2}, we have
    $r_\ast'(0) = r_\ast'(1) = 0$, and the boundary conditions give
    $z_\ast(0) = z_\ast(1) = 0$.
    Therefore, by using the transformation $t\mapsto 2-t$ and $z\mapsto -z$, the curve can be extended smoothly.
    
    Let $\tilde{\gamma}_\ast(s) = \gamma_\ast(s/\sqrt{\lambda})$.
    Then, by \eqref{eq:lem1_1}, \eqref{eq:lem1_3}, the definition of $\lambda$, and the fact that $V=\hat{V}$ on $\bar{\Omega}$, we have, for all $s \in [0, 2\sqrt{\lambda}]$,
    \begin{align*}
        &\tilde{\gamma}_\ast''(s) + \nabla V(\tilde{\gamma}_\ast(s))
        = \frac 1 \lambda \gamma_\ast''(s/\sqrt{\lambda})
        + \nabla V(\gamma_\ast(s/\sqrt \lambda)) = 0,\\
        &\frac12 |\tilde\gamma_\ast'(s)|^2 + V(\tilde \gamma_\ast(s))
        = \frac 1 {2\lambda} |\gamma_\ast'(s/\sqrt \lambda)|^2
        + V(\gamma_\ast(s/\sqrt \lambda)) = h.
    \end{align*}
    This proves the claim.
\end{proof}

\subsection{Some Lemmas to apply the linking theorem}
In this section, we prove several lemmas needed to apply Theorem \ref{thm:main}.

\begin{lem}\label{lem:boundary}
    Suppose that $\{\gamma_k\}_{k=1}^\infty\subset \hat{\Gamma}$ converges weakly to
    $\bar{\gamma} \in \partial \hat{\Gamma}$. Then
    \begin{equation*}
        \lim_{k \to \infty} B(\gamma_k) =  \lim_{k \to \infty} \int_0^1 ( h - \hat V(\gamma_k(t)))dt = -\infty.
    \end{equation*}
\end{lem}

\begin{proof}
    Since $\bar{\gamma} \in \partial \hat{\Gamma}$, there exists $t_0 \in [0,1]$
    such that $\bar{\gamma}(t_0) \in \partial \hat{\Omega}$.
    By the Sobolev embedding theorem, $H^1$ is compactly embedded into the space of continuous functions.
    Hence, since $\gamma_k$ converges weakly to $\bar{\gamma}$ in $H^1$, it converges uniformly to $\bar{\gamma}$.
    In particular, $\gamma_k(t_0) \to \bar{\gamma}(t_0)$.

    By the boundedness of weakly convergent sequences, there exists $M_1 > 0$ such that
    $\|\gamma_k\|_X \leq M_1$.
    If $t > t_0$, then by the Cauchy--Schwarz inequality, for every $k \in \mathbb{N}$,
    \begin{align*}
        |\gamma_k(t) - \gamma_k(t_0)|
        &\leq \int_{t_0}^t|\gamma_k'(\tau)| d\tau \\
        &\leq |t - t_0|^{1/2}
        \left(\int_0^1 |\gamma_k'(\tau)|^2 d\tau \right)^{1/2}\\
        &\leq M_1|t - t_0|^{1/2}.
    \end{align*}
    The case $t<t_0$ is treated in the same way.
    
    Moreover, since $V$ is Lipschitz continuous on $\bar{\hat{\Omega}}$ (if necessary we replace $\varepsilon$ with smaller one), there exists $M_2>0$ such that,
    for every $k \in \mathbb{N}$ and $t \in [0,1]$,
    \begin{align*}
        |V(\gamma_k(t)) - V(\gamma_k(t_0))|
        &\leq  M_2 |\gamma_k(t) - \gamma_k(t_0)|\\
        & \leq M |t - t_0|^{1/2},
    \end{align*}
    where $M := M_1 M_2$.
    Therefore,
    \begin{equation}\label{eq:lem_boundary_1}
        | h + 2 \varepsilon - V(\gamma_k(t))|
        \leq | h + 2 \varepsilon -V(\gamma_k(t_0))| + M |t - t_0|^{1/2}.
    \end{equation}
    Let $ m := \inf_{\Omega} V = \inf_{\hat{\Omega}}\hat{V} < h$.
    Suppose that $t_0 \in (0,1)$. Then, for $t_0 < t < 1$, we have
    \begin{align*}
        \int_0^1 ( h - \hat{V}(\gamma_k(t)))\,dt
        &= \int_0^{t_0} ( h - \hat{V}(\gamma_k(\tau)))\,d \tau
        + \int_{t_0}^t ( h - \hat{V}(\gamma_k(\tau)))\,d\tau \\
        &\quad + \int_t^1 ( h - \hat{V}(\gamma_k(\tau)))\,d \tau\\
        & \leq h-m+ \int_{t_0}^t ( h - \hat{V}(\gamma_k(\tau)))\,d \tau.
    \end{align*}
    If $|t - t_0|$ is chosen sufficiently small, then by the continuity of $\gamma_k$,
    for all sufficiently large $k$,
    \begin{equation*}
        V(\gamma_k(\tau)) \in \left[h + \frac{3 \varepsilon}{2}, h + 2\varepsilon \right)
        \quad \forall \tau \in [t_0, t].
    \end{equation*}
    Hence, noting from the choice of $t_0$ that
    \begin{equation*}
        b_k :=| h + 2\varepsilon - V(\gamma_k(t_0))| \to 0
        \quad (k \to \infty),
    \end{equation*}
    and using \eqref{eq:lem_boundary_1}, we obtain
    \begin{align*}
        \int_0^1( h - \hat{V}(\gamma_k(t)))\,dt
        &\leq (h - m) + h|t - t_0|
        - \int_{t_0}^t\frac{d \tau}{(V(\gamma_k(\tau)) - h -2 \varepsilon)^2}\\
        & \leq (h - m ) + h|t - t_0|
        - \int_{t_0}^t \frac{d\tau}{(b_k + M |\tau - t_0|^{1/2})^2}\\
        &= (h - m) + h|t - t_0| \\
        &\quad - \frac{2}{M^2}\left\{
        \log \left( \frac{b_k + M \sqrt{t - t_0}}{b_k} \right)
        +\frac{b_k}{b_k + M \sqrt{t - t_0}} - 1
        \right\}\\
        &\to -\infty \quad (k \to \infty).
    \end{align*}
    If $t_0 = 0$, the integral over $[0,t_0]$ is absent.
    If $t_0 = 1$, we take $t<t_0$ and split the integral over $[0,t]$ and $[t,t_0]$ and the same estimate applies.
\end{proof}

\begin{lem}\label{lem:ps}
    The functional $\hat{I}$ satisfies the Palais--Smale condition (PS) on $(0,\infty)$.
\end{lem}

\begin{proof}
    Let $c>0$, and suppose that $\{\gamma_k\} \subset \hat{\Gamma}$ satisfies
    $\hat{I}(\gamma_k) \to c$ and $\|\hat{I}'(\gamma_k)\|_{X\!^\ast} \to 0$.
    We show that $\{\gamma_k\}$ has a strongly convergent subsequence.
    First, assume that
    \begin{equation*}
        \sup_k A(\gamma_k)
        = \sup_k \int_0^1 |\gamma_k'|^2dt < + \infty,
    \end{equation*}
    and show that a strongly convergent subsequence exists.
    
    Together with the boundedness of $\hat{\Omega}$, this implies that $\{\gamma_k\}$
    is bounded. Hence, after passing to a subsequence, we may assume that
    $\gamma_k \rightharpoonup \bar{\gamma}$ weakly.
    Since
    \begin{equation*}
        \hat I(\gamma_k) = \frac 12A(\gamma_k) B(\gamma_k)
        \to c \in(0,\infty)\quad (k \to \infty),
    \end{equation*}
    the sequence $B(\gamma_k)$ is bounded from below.
    Therefore, by Lemma \ref{lem:boundary}, we have $\bar{\gamma} \in \hat{\Gamma}$.
    Since $\|\hat I'(\gamma_k)\|_{X\!^\ast} \to 0$, using the variation in the direction
    $\gamma_k - \bar{\gamma}$ gives
    \begin{align*}
         \int_0^1 \langle \gamma_k', \gamma_k' - {\bar{\gamma}}' \rangle dt
         \int_0^1 (h - \hat V (\gamma_k))dt
         - \frac12 \int_0^1 |\gamma_k'|^2 dt
         \int_0^1 \langle \nabla  \hat V(\gamma_k), \gamma_k - \bar{\gamma} \rangle dt
         \to 0.
    \end{align*}
    By the Sobolev embedding theorem, $\gamma_k$ converges uniformly to $\bar{\gamma}$.
    Moreover, since $\bar{\gamma} \in \hat{\Gamma}$, there exists a compact subset contained
    in the interior of $\hat{\Omega}$ such that, for all sufficiently large $k$,
    $\gamma_k([0,1])$ is contained in this compact set.
    Therefore, by the continuity of $\nabla \hat{V}$, we have
    $\nabla \hat{V}(\gamma_k) \to \nabla \hat{V}(\bar{\gamma})$ uniformly.
    Hence
    \begin{equation*}
        \int_0^1 \langle \nabla \hat{V}(\gamma_k), \gamma_k - \bar{\gamma} \rangle dt \to 0
        \quad (k \to \infty).
    \end{equation*}
    Thus,
    \begin{equation*}
        \int_0^1 \langle \gamma_k', \gamma_k' - \bar{\gamma}' \rangle dt
        \int_0^1 (h - \hat V (\gamma_k))dt \to 0
        \quad (k \to \infty).
    \end{equation*}
    Since $\hat{I}(\gamma_k) \to c>0$ and $A(\gamma_k)$ is bounded,
    $B(\gamma_k)$ is bounded uniformly from below by a positive constant for all sufficiently large $k$.
    Hence
    \begin{equation*}
        \int_0^1 \langle \gamma_k', \gamma_k' - \bar{\gamma}' \rangle dt \to 0.
    \end{equation*}
    Combining this with weak convergence, we obtain
    \begin{align*}
        \int_0^1 |\gamma_k' - \bar{\gamma}'|^2dt
        &= \int_0^1 \langle \gamma_k', \gamma_k' - \bar{\gamma}' \rangle dt
        - \int_0^1 \langle \bar{\gamma}', \gamma_k' - \bar{\gamma}' \rangle dt
        \to 0.
    \end{align*}
    Moreover, $\|\gamma_k - \bar \gamma\|_{L^2} \to 0$ follows from uniform convergence.
    Therefore, $\gamma_k$ converges strongly to $\bar{\gamma}$.

    It remains to prove that
    \begin{equation*}
        \sup_k \int_0^1 |\gamma_k'|^2dt < + \infty.
    \end{equation*}
    To this end, suppose, for contradiction, that
    \begin{equation*}
        \lim_{k \to \infty} \int_0^1 |\gamma_k'|^2 dt = +\infty.
    \end{equation*}
    Since $\|\hat{I}'(\gamma_k)\|_{X\!^\ast} \to 0$, for every $\gamma = (r,z) \in X$ we have
    \begin{align*}
         &\left|
         \int_0^1 \langle \gamma_k', \gamma' \rangle dt
         \int_0^1 (h - \hat V (\gamma_k))dt
         - \frac 12 \int_0^1 |\gamma_k'|^2 dt
         \int_0^1 \langle \nabla \hat V(\gamma_k), \gamma \rangle dt
         \right| \\
         &\leq  \|\hat{I}'(\gamma_k)\|_{X\!^\ast} \|\gamma\|_X .
    \end{align*}
    Let $K$ be an upper bound for the absolute values of the eigenvalues of the Hessian matrix of $V$
    on $\bar{\hat{\Omega}}$, that is, set
    \begin{equation*}
        K:= \sup_{x \in \hat{\Omega}, y \neq0}
        \frac{|\langle \nabla^2 V(x)y, y \rangle|}{|y|^2}< +\infty.
    \end{equation*}
     
    Since $V(r,z) = V(r,-z)$, we have $V_z(r,0) = 0$.
    Hence we may take $\gamma(t) = \nabla V(\gamma_k(t)) \in X$, and obtain
    \begin{equation}\label{eq:lem2_1}
        \begin{split}
        &\int_0^1 |\gamma_k'|^2dt
        \int_0^1 \langle \nabla \hat V(\gamma_k), \nabla V(\gamma_k) \rangle dt\\
        \leq& \|\hat{I}'(\gamma_k)\|_{X\!^\ast} \| \nabla V(\gamma_k)\|_X
        + 2K \int_0^1 |\gamma_k'|^2 dt
        \int_0^1(h - \hat{V}(\gamma_k))dt\\
        =& \|\hat{I}'(\gamma_k)\|_{X\!^\ast} \| \nabla V(\gamma_k)\|_X + 2K\tilde{c}.
        \end{split}
    \end{equation}
    where $c < \tilde c < +\infty$ can be chosen since $\hat{I}(\gamma_k) \to c$.
    
    By the boundedness of $\nabla V$ on $\hat{\Omega}$, there exists $C>0$ such that
    \begin{align*}
        \|\nabla V(\gamma_k)\|_X
        &\leq \left(
        \int_0^1
        \langle \nabla^2 V(\gamma_k)\gamma_k',
        \nabla^2V(\gamma_k)\gamma_k '\rangle dt
        + \int_0^1 \langle \nabla V(\gamma_k), \nabla V(\gamma_k) \rangle dt
        \right)^{1/2}\\
        &\leq \left( K^2\int_0^1 |\gamma_k'|^2 dt + C \right)^{1/2}.
    \end{align*}
    Therefore, using $\|\hat{I}'(\gamma_k)\| \to 0$ and dividing both sides of
    \eqref{eq:lem2_1} by $\int_0^1|\gamma_k'|^2dt$, we obtain
    \begin{equation*}
        \lim_{k \to \infty}
        \int_0^1 \langle \nabla \hat{V}(\gamma_k), \nabla V(\gamma_k) \rangle dt = 0.
    \end{equation*}
    Moreover,
    \begin{align*}
        \langle\nabla \hat{V}(r,z), \nabla V(r,z) \rangle
        = \varphi'(V(r,z)) |\nabla V(r,z)|^2
        \geq |\nabla V(r,z)|^2.
    \end{align*}
    Hence
    \begin{align*}
        \lim_{k \to \infty} \int_0^1 |\nabla V(\gamma_k)|^2 dt = 0.
    \end{align*}
    Passing to a subsequence, we get
    \begin{equation}\label{eq:lem1_conv}
        \lim_{k \to \infty} |\nabla V(\gamma_k(t))| = 0
        \quad \text{for a.e. } t \in [0,1].
    \end{equation}

    Set
    \begin{equation*}
        E_k := \{ t \in [0,1] \mid V(\gamma_k(t)) \in [h+\varepsilon, h + 2\varepsilon)\}.
    \end{equation*}
    By assumption, $\nabla V \neq 0$ on $E_k$.
    It follows that there exists $m_\varepsilon > 0$ such that $|\nabla V(r,z)| \geq m_\varepsilon$ on the set $\{(r,z) \mid V(r,z) \in [h+\varepsilon, h + 2 \varepsilon)\}$.
    For $t \in \{t\ \mid \  V(\gamma_k(t)) > h + \frac{3}{2}\varepsilon \}$,
    we have
    \begin{align*}
        \langle \nabla \hat{V}(\gamma_k(t)), \nabla V(\gamma_k(t)) \rangle
        &= \varphi'(V(\gamma_k(t))) | \nabla V(\gamma_k(t))|^2\\
        & \geq \frac{2 m_\varepsilon^2}{(h + 2 \varepsilon - V(\gamma_k(t)))^3}\\
        & \geq \frac{2m_\varepsilon^2}{ h + 2\varepsilon - V(\gamma_k(t))}
        \hat{V}(\gamma_k(t))\\
        &\geq \frac{4 m_\varepsilon^2}{\varepsilon} \hat{V}(\gamma_k(t)).
    \end{align*}
    Hence
    \begin{align*}
        \int_{E_k \cap \{t\mid V(\gamma_k(t)) > h + \frac{3}{2}\varepsilon\}}
        \hat{V} (\gamma_k)\, dt
        &\leq \frac{\varepsilon}{4 m_\varepsilon^2}
        \int_0^1
        \langle \nabla \hat{V}(\gamma_k), \nabla V (\gamma_k) \rangle\, dt
        \to 0
        \quad (k \to \infty).
    \end{align*}
    On the other hand, $\hat{V}$ is bounded on
    $\{t\mid V(\gamma_k(t)) \in [h + \varepsilon, h + \frac{3}{2}\varepsilon]\}$,
    and $|E_k| \to 0$. Therefore,
    \begin{equation*}
        \int_{\{t\mid V(\gamma_k(t)) \in [h + \varepsilon, h + \frac{3}{2}\varepsilon]\}}
        \hat{V}(\gamma_k)dt \to 0.
    \end{equation*}
    Consequently,
    \begin{equation*}
        \int_{E_k} \hat{V}(\gamma_k) dt\to 0.
    \end{equation*}
    On $[0,1] \setminus E_k$, we have
    $V(\gamma_k(t)) = \hat{V}(\gamma_k(t))$.
    Since $\nabla V(r,z) \neq 0$ on $\partial \Omega$, there exists $\delta>0$
    such that $\nabla V(r,z) \neq 0$ whenever $V(r,z) \in (h - \delta, h)$.
    Hence, by \eqref{eq:lem1_conv}, if
    \begin{equation*}
        F_k := \{t \in [0,1] \mid V(\gamma_k(t)) < h - \delta/2\},
    \end{equation*}
    then $|F_k| \to 1$.
    Moreover,
    \begin{align*}
        &\limsup_{k \to \infty}
        \int_{[0,1] \setminus (E_k \sqcup F_k)}
        \hat{V}(\gamma_k(t))dt
        \leq \limsup_{k \to \infty}
        (1 - |E_k| -| F_k|)( h+\varepsilon)  \to 0,\\
        &\limsup_{k \to \infty}
        \int_{F_k} \hat{V}(\gamma_k(t))dt
        \leq h - \delta/2.
    \end{align*}
    Combining the above estimates, we obtain
    \begin{align*}
        &\liminf_{k \to \infty}
        \int_0^1 (h - \hat{V}(\gamma_k(t)))dt\\
        \geq&\ h
        - \limsup_{k \to \infty}
        \int_{[0,1] \setminus (E_k \sqcup F_k)}
        \hat{V}(\gamma_k(t))dt
        - \limsup_{k \to \infty}
        \int_{F_k} \hat{V}(\gamma_k(t))dt \\
        &\quad
        -\limsup_{k \to \infty}
        \int_{E_k} \hat{V}(\gamma_k(t))dt
        \geq \delta/2.
    \end{align*}
    On the other hand, since $\hat{I}(\gamma_k) \to c \in(0,\infty)$ and
    $\int_0^1|\gamma_k'|^2dt \to \infty$, we have
    \begin{equation*}
        \lim_{k \to \infty}
        \int_0^1 (h - \hat{V}(\gamma_k)) dt = 0,
    \end{equation*}
    which is a contradiction.
\end{proof}

\subsection{Proof of Theorem \ref{thm:main}}
We introduce a new coordinate $\tilde{r} = r - r^\ast$ using a point $(r^\ast,0)$ in the interior of the Hill region.
Accordingly, we also shift $\hat{\Gamma}$, $\hat{\Omega}$, and $\hat{I}$ by $r^\ast$,
while keeping the same notation.

By Lemma \ref{lem:critical_pt}, it suffices to find a positive critical value of $\hat{I}$.
Condition (I3) was proved in Lemma \ref{lem:ps}.
Condition (I1) follows from Lemma \ref{lem:boundary} and the inequality $\int_0^1 |\gamma'|^2dt \geq 0$.

Let
\begin{equation*}
    E_1 = \{\text{constant curves in } X\}.
\end{equation*}
Then
\begin{equation*}
    E_1^\perp
    =
    \left\{
        \tilde{\gamma}(t) = (\tilde{r}(t), z(t)) \in X
        \ \middle|\ 
        \int_0^1 \tilde r(t)dt = 0
    \right\}.
\end{equation*}
On $E_1 \cap \hat{\Gamma}$, the functional $\hat I$ is clearly equal to $0$.
Thus condition (I2) (i) is satisfied.

Take $\rho>0$ sufficiently small and consider
\begin{equation*}
    \partial \bar{B}_\rho \cap E_1^\perp.
\end{equation*}
For $\tilde{\gamma} \in \partial \bar{B}_\rho \cap E_1^\perp$, $\|\tilde \gamma\|_{L^\infty}$ is also sufficiently small.
Therefore, for curves in $\partial \bar{B}_\rho \cap E_1^\perp \subset \hat{\Gamma}$,
$V$ is uniformly bounded above by a constant strictly smaller than $h$.
In other words, $B(\tilde{\gamma})$ is bounded from below by a positive constant.

Moreover, since $\tilde \gamma \in E_1^\perp$, Poincar\'e's inequality can be applied.
That is, there exists $C>0$ such that
\begin{equation*}
    \| \tilde{\gamma}\|_{L^2} \leq C \|\tilde{\gamma}'\|_{L^2}.
\end{equation*}
Thus $A(\tilde{\gamma})$ is also uniformly bounded from below by a positive constant.
Consequently,
\begin{equation*}
    \inf_{\partial B_\rho \cap E_1^\perp} \hat{I} =: \alpha > 0,
\end{equation*}
and condition (ii) of (I2) is satisfied.

Let $z_0(t) = \sin \pi t$ and define
\begin{equation*}
    Q_{\hat{\Gamma}}
    =
    \{ (\tilde r, sz_0(t)) \mid \tilde r \in \mathbb{R},\ s \geq 0\}
    \cap \hat{\Gamma}.
\end{equation*}
In other words, we take $e = (0,z_0)$.
The boundedness of $Q_{\hat{\Gamma}}$ is clear from the compactness of the Hill region.

For $\tilde{\gamma}=(\tilde r,s z_0(t)) \in Q_{\hat{\Gamma}}$, the boundedness of $\hat{\Omega}$
implies that there exists $s_0>0$ such that $s < s_0$.
Hence, using $m=\min_\Omega V$, we obtain
\begin{align*}
    \hat{I}(\tilde{\gamma})
    &= \frac 12\int_0^1 s^2| z_0'(t)|^2 dt
    \int_0^1 (h - \hat V(r,sz_0(t))) dt\\
    &\leq \frac{s_0^2 \pi^2}{4} (h - m) < +\infty.
\end{align*}
The fact that $Q_{\hat{\Gamma}}$ is homeomorphic to
$(B_1 \cap E_1) \times [0,1)$ follows from (V3).
Indeed, $(r,sz_0) \in \hat{\Gamma}$ is equivalent to
\begin{equation*}
    r \in J,\quad 0 \leq s < \zeta(r),
\end{equation*}
and hence one can define a homeomorphism
\begin{equation*}
    J \times [0,1) \ni (r,\sigma)
    \mapsto
    (r, \sigma \zeta(r)z_0).
\end{equation*}
Furthermore, $J$ is homeomorphic to $B_1 \cap E_1$.

Thus condition (iii) of (I2) is also satisfied.
Therefore, the existence of a critical point with positive critical value follows.

\subsection{Morse index of the critical point}

The critical point obtained in Theorem \ref{thm:main} may be degenerate.
Nevertheless, even allowing for this possibility, its Morse index is at most
\begin{equation*}
    \dim(E_1)+1=2.
\end{equation*}
This follows from Corollary 10.5 of \cite{Ghoussoub93}.
Although the original result is stated in a more general setting involving, for instance, group actions,
we state here a simpler form sufficient for our purposes.
Moreover, we note that since the boundary of the domain of definition of the functional do not affect the minimax level the statement verifies.

Let $X$ be a Hilbert space, and let $I$ be a $C^2$ functional on $X$.

\begin{dfn}
    Let $X$ be a Hilbert space and $B \subset X$ be a closed set.
    We say that a family $\mathcal{F}$ is homotopic of dimension $n$ with boundary $B$ if there exist a compact subset $D$ of $\mathbb{R}^n$, containing a closed subset $D_0$, and a continuous map $\sigma:D_0 \to B$ such that
    \begin{equation*}
        \mathcal{F} = \{ A \subset X \mid A = f(D),\ \text{for some } f \in C(D,X)\text{ with } f|_{D_0}= \sigma \}.
    \end{equation*}
\end{dfn}

Define
\begin{equation*}
    c(I,\mathcal{F}) := \inf_{A \in \mathcal{F}} \max_{\gamma \in A} I(\gamma).
\end{equation*}

\begin{prop}
    Let $X$ be a Hilbert space, let $B \subset X$ be a closed set, and let
    $\mathcal{F}$ be a homotopic family of dimension $n$ with boundary $B$.
    Set $c = c(I,\mathcal{F})$, and assume that $\sup_B I < c$.
    Suppose that $I$ satisfies the Palais--Smale condition at level $c$ and that
    $I''$ is Fredholm on
    \begin{equation*}
        K_c:=\{ \gamma \in X \mid I(\gamma) = c,\ I'(\gamma) = 0\}.
    \end{equation*}
    Then there exists $\gamma \in K_c$ such that
    \begin{equation*}
        m(\gamma) \leq n.
    \end{equation*}
    where $m(\cdot)$ denotes the Morse index.
\end{prop}

In the present case, we take
\begin{equation*}
    D=Q,\quad D_0 = \partial Q,\quad \sigma = \mathrm{id}.
\end{equation*}

Indeed, since the Palais--Smale condition has already been proved, it remains only to show that
$\hat{I}''$ is a Fredholm operator at positive critical points.

Let $\gamma_\ast \in X$ be the critical point obtained in Theorem \ref{thm:main}, with
$\hat{I}(\gamma_\ast)>0$.
Then
\begin{equation*}
    A(\gamma_\ast)>0,\quad B(\gamma_\ast)>0.
\end{equation*}
Moreover, by Lemma \ref{lem:critical_pt}, we have $\gamma_\ast([0,1]) \subset \bar{\Omega}$.
For $\xi_1=(r_1,z_1), \xi_2=(r_2,z_2) \in X$, we have
\begin{align*}
    &\hat{I}''(\gamma_\ast)[\xi_1,\xi_2] \\
    =& B(\gamma_\ast)\int_0^1 \langle \xi_1', \xi_2'\rangle \,dt
    -  \int_0^1 \langle \gamma_\ast', \xi_1' \rangle \,dt
    \int_0^1 \langle \nabla \hat V (\gamma_\ast), \xi_2 \rangle \,dt \\
    &-  \int_0^1 \langle \gamma_\ast' , \xi_2' \rangle \,dt
    \int_0^1 \langle \nabla \hat V (\gamma_\ast), \xi_1 \rangle \,dt
    - \frac12A(\gamma_\ast) \int_0^1 \hat V''(\gamma_\ast)[\xi_1, \xi_2] \,dt.
\end{align*}
Here we also denote by $\hat{I}''$ the bilinear form corresponding to the Hessian.
We write
\begin{equation*}
    \hat{I}''(\gamma_\ast)= B(\gamma_\ast) L + \frac12K.
\end{equation*}
That is,
\begin{align*}
    \langle L \xi_1, \xi_2 \rangle_X
    &= \int_0^1 \langle \xi_1', \xi_2' \rangle \,dt,\\
    \langle K \xi_1, \xi_2 \rangle_X
    &= - 2 \int_0^1 \langle \gamma_\ast', \xi_1' \rangle \,dt
    \int_0^1 \langle \nabla \hat V (\gamma_\ast), \xi_2 \rangle \,dt \\
    &\quad - 2 \int_0^1 \langle \gamma_\ast' , \xi_2' \rangle \,dt
    \int_0^1 \langle \nabla \hat V (\gamma_\ast), \xi_1 \rangle \,dt \\
    &\quad - A(\gamma_\ast) \int_0^1 \hat V''(\gamma_\ast)[\xi_1, \xi_2] \,dt.
\end{align*}

For $L$, recalling the definition of the inner product on $X$, we have
\begin{align*}
    \langle L\xi_1,\xi_2 \rangle_X
    &= \int_0^1\langle \xi_1', \xi_2' \rangle \,dt \\
    &= \int_0^1 r_1' r_2' \,dt + \int_0^1 z_1' z_2' \,dt \\
    &= \langle \xi_1, \xi_2 \rangle_X
    - \int_0^1 (r_1 r_2+ z_1z_2)\,dt.
\end{align*}
By the Rellich--Kondrachov theorem, the embedding
\begin{equation*}
    H^1([0,1], \mathbb{R}) \hookrightarrow L^2([0,1], \mathbb{R})
\end{equation*}
is compact.
Thus $L$ is the sum of the identity operator and a compact operator, and hence $L$ is Fredholm.

We next consider $K$.
The first and second terms in $K$ are products of linear functionals, and hence define finite-rank operators.
In particular, they are compact.
The third term in $K$ is also compact.
Indeed, since $\hat{V}''$ is bounded on $\bar{\Omega}$ and
\begin{equation*}
    \gamma_\ast([0,1]) \subset \bar{\Omega},
\end{equation*}
we have
\begin{equation*}
    \hat{V}''(\gamma_\ast(\cdot)) \in L^\infty([0,1], \mathcal{L}(\mathbb{R},\mathbb{R})).
\end{equation*}
Using again the compactness of the embedding $H^1 \hookrightarrow L^2$, together with the boundedness of its adjoint, we see that the third term is compact as well.

Consequently, $\hat{I}''$ is Fredholm.

\section{Morse index of the planar solution}

In this section, we consider whether the critical point obtained in Section 2 can coincide with the planar solutions.
In practice, this has to be checked by computations in each application.
However, we note that if the Morse index of the planar orbit is at least $3$, then it cannot coincide with the critical point obtained in Section 2.

We denote the half-period of a planar orbit by
\begin{equation*}
    \gamma_0(t) = (r_0(t),0) \in X.
\end{equation*}
Suppose that this orbit is obtained as a critical point of $\hat I$.
If it is a constant curve, then clearly $\hat I(\gamma_0)=0$, and hence it cannot coincide with the critical point obtained in Theorem \ref{thm:main}.

Assume that $\gamma_0$ is a nonconstant critical point satisfying $\hat{I}(\gamma_0)>0$.
By Lemma \ref{lem:critical_pt}, we have
\begin{equation*}
    r_0'(0)=r_0'(1)=0,
    \qquad
    \gamma_0([0,1]) \subset \bar{\Omega}.
\end{equation*}
Thus $\gamma_0(0), \gamma_0(1) \in \partial \Omega \cap \{z=0\}$, and $\gamma_0$ is a brake orbit.
We may assume that
\begin{equation*}
    r_0(0) < r_0(1),
    \qquad
    r_0'(t)>0 \quad \forall t \in (0,1).
\end{equation*}
Indeed, if $r_0(0)>r_0(1)$, then by reversing time we obtain the above situation without changing either the Morse index or the value of the functional.

If there exists $t_0 \in (0,1)$ such that $r_0'(t_0)=0$, then the restriction of the orbit to $[0,t_0]$ can be brought into the above situation by a change of time.
The same argument applies to the restriction to $[t_0,1]$; after further decomposing it if necessary and reversing time when needed, each piece becomes a brake orbit of the above type.
If a planar brake orbit is a concatenation of several brake arcs, then variations supported in one arc give a negative subspace of the same dimension as for that arc. Hence the Morse index of the concatenated orbit is at least the Morse index of each component brake arc. 

In what follows, since $\gamma_0([0,1]) \subset \bar{\Omega}$, we write everything in terms of $V$ rather than $\hat V$.
Set
\begin{equation*}
    \lambda_0
    :=
    \frac{A(\gamma_0)}{2B(\gamma_0)}
    =
    \frac{\int_0^1 |r_0'(t)|^2 \,dt}
    {2 \int_0^1 (h - V(r_0(t),0))\,dt}.
\end{equation*}
By Lemma \ref{lem:critical_pt}, if we introduce the physical time
\begin{equation*}
    s \in [0,\sqrt{\lambda_0}]
\end{equation*}
and define
\begin{equation*}
    \tilde r_0(s) = r_0(s/\sqrt{\lambda_0}),
\end{equation*}
then
\begin{align*}
    &\tilde r_0''(s) = -\frac{\partial V}{\partial r}(\tilde r_0(s)),\\
    &\frac12  \tilde r_0'(s)^2 + V( \tilde r_0(s)) = h,\\
    &\tilde r_0'(0) = \tilde r_0'(\sqrt{\lambda_0})=0,\quad
    \tilde r_0'(s) > 0\quad \forall s \in (0, \sqrt{\lambda_0}).
\end{align*}

By (V1), we have
\begin{equation*}
    \hat V_z (r,0) =  \hat V_{rz}(r,0) = 0.
\end{equation*}
Therefore, if we take variations
\begin{equation*}
    \xi=(u,0),\qquad \eta = (0,v) \in X,
\end{equation*}
then
\begin{align*}
    A''(\gamma_0)[\xi, \eta]
    &= 2\int_0^1 \langle \xi', \eta' \rangle \,dt = 0,\\
    B''(\gamma_0)[\xi, \eta]
    &= \int_0^1 V_{rz}(\gamma_0(t))u(t)v(t)\,dt = 0.
\end{align*}
Moreover, for the terms $A'(\gamma_0)[\xi] B'(\gamma_0)[\eta],\ A'(\gamma_0)[\eta] B'(\gamma_0)[\xi]$, we have
\begin{align*}
    B'(\gamma_0)[\eta]
    &= -\int_0^1 V_z(\gamma_0(t))v(t)\,dt = 0,\\
    A'(\gamma_0)[\eta]
    &= 2\int_0^1 \langle \gamma_0', \eta' \rangle \,dt = 0.
\end{align*}
Hence these terms also vanish.
Consequently,
\begin{equation*}
    \hat{I}''(\gamma_0)[(u,v),(u,v)]
    =
    \hat{I}''(\gamma_0)[(u,0),(u,0)]
    +
    \hat{I}''(\gamma_0)[(0,v),(0,v)].
\end{equation*}
Thus it suffices to study the Morse indices in the $r$-direction and the $z$-direction separately.

\subsection{Morse index in the $z$-direction}

We write variations in the $z$-direction as $\eta = (0,v),\ \eta_1 = (0,v_1),\ \eta_2 = (0,v_2) \in X$.
Then
\begin{align*}
    \hat I ''(\gamma_0)[\eta_1, \eta_2]
    &= \frac 12 A''(\gamma_0)[\eta_1, \eta_2]B(\gamma_0)
    + \frac12 A(\gamma_0)B''(\gamma_0)[\eta_1, \eta_2]\\
    &= B(\gamma_0) \int_0^1 v_1'(t) v_2'(t) \,dt
    - \frac12  A(\gamma_0)
    \int_0^1 V_{zz}(\gamma_0(t)) v_1(t) v_2(t) \,dt.
\end{align*}

We now change the variable of integration to the angular variable $\theta$.
We denote by $\bar v_1(\theta),\ \bar v_2(\theta) \in H_0^1([0,\theta_0], \mathbb{R}), \bar r_0(\theta) \in H^1([0, \theta_0], \mathbb{R})$ the functions obtained from $v_1(t)$, $v_2(t)$, and $r_0(t)$ after this change of variables where
\begin{equation*}
    \theta_0 
    =
    \int_0^{\sqrt{\lambda_0}} \frac{d\theta}{ds}\, ds
    =
    \int_0^{\sqrt{\lambda_0}}
    \frac{\omega}{\tilde r_0(s)^2}\, ds
    =
    \int_0^1 \frac{\omega \sqrt{\lambda_0}}{r_0(t)^2}\, dt
\end{equation*}
Using
\begin{equation*}
    \omega = r^2 \frac{d\theta}{ds},
    \qquad
    s = \sqrt{\lambda_0}t,
\end{equation*}
we obtain
\begin{align*}
    &\hat I ''(\gamma_0)[\eta_1, \eta_2]\\
    =& B(\gamma_0) \omega \sqrt{\lambda_0}
    \int_0^{\theta_0}
    \frac{1}{\bar r_0(\theta)^2}
    \bar v_1'(\theta) \bar v_2'(\theta)\, d \theta
    - \frac{1}{2\omega \sqrt{\lambda_0}} A(\gamma_0)
    \int_0^{\theta_0}
    \bar r_0(\theta)^2
    V_{zz}(\bar \gamma_0(\theta))
    \bar v_1(\theta) \bar v_2(\theta)\, d \theta\\
    =& \frac 1 \omega
    \sqrt{\frac{A(\gamma_0) B(\gamma_0)}{2}}
    \int_0^{\theta_0}
    \left(
        \frac{\omega^2}{\bar r_0(\theta)^2}
        \bar v_1'(\theta)\bar v_2'(\theta)
        -
        \bar r_0(\theta)^2
        V_{zz}(\bar \gamma_0(\theta))
        \bar v_1(\theta) \bar v_2(\theta)
    \right)d \theta\\
    =& \frac 1 \omega
    \sqrt{\frac{A(\gamma_0) B(\gamma_0)}{2}}
    \int_0^{\theta_0}
    \left[
        -\frac{d}{d \theta}
        \left(
            \frac{\omega^2}{\bar r_0(\theta)^2}
            \bar v_1'(\theta)
        \right)
        -
        \bar r_0(\theta)^2
        V_{zz}(\bar \gamma_0(\theta))
        \bar v_1(\theta)
    \right]
    \bar v_2(\theta)\, d \theta.
\end{align*}
Therefore, the sign is determined by
\begin{equation}\label{eq:z_second_variation}
    Q_z(\bar v)
    :=
    \int_0^{\theta_0}
    \left(
        \frac{\omega^2}{\bar r_0(\theta)^2} \bar v'(\theta)^2
        -
        \bar r_0(\theta)^2
        V_{zz}(\bar r_0(\theta),0)\bar v(\theta)^2
    \right)d\theta.
\end{equation}
Hence we have the following:

\begin{lem}\label{lem:z_morse}
    The Morse index of the planar solution $\gamma_0$ in the $z$-direction coincides with the number of negative eigenvalues of the Sturm--Liouville operator
    \begin{equation*}
        \mathcal L_z \bar v
        :=
        -\frac{d}{d\theta}
        \left(
            \frac{\omega^2}{\bar r_0(\theta)^2}\bar v'(\theta)
        \right)
        -
        \bar r_0(\theta)^2
        V_{zz}(\bar r_0(\theta),0)\bar v(\theta),
    \end{equation*}
    under the Dirichlet boundary condition
    \begin{equation*}
        \bar v(0)=\bar v(\theta_0)=0.
    \end{equation*}
\end{lem}

\subsection{Morse index of $r$-direction}

We next consider variations in the $r$-direction.
Let $\xi=(u,0),\ \xi_1=(u_1,0),\ \xi_2=(u_2,0)\in X$.
Then
\begin{align*}
    \hat I''(\gamma_0)[\xi_1,\xi_2]
    =& B(\gamma_0)\int_0^1 u_1'(t)u_2'(t)\,dt \\
    & - \left( \int_0^1 r_0'(t)u_1'(t)\,dt \right)
    \left(\int_0^1 V_r(r_0(t),0)u_2(t)\,dt \right) \\
    & - \left( \int_0^1 r_0'(t)u_2'(t)\,dt \right)
    \left( \int_0^1 V_r(r_0(t),0)u_1(t)\,dt \right) \\
    &-\frac12 A(\gamma_0)
    \int_0^1 V_{rr}(r_0(t),0)u_1(t)u_2(t)\,dt .
\end{align*}

Since $\gamma_0$ is a critical point, we have
\begin{equation*}
    B(\gamma_0) r_0''(t) + \frac12 A(\gamma_0)V_r(r_0(t),0)=0.
\end{equation*}
Moreover, since $r_0'(0)=r_0'(1)=0$, for every $u\in H^1([0,1], \mathbb{R})$ we obtain
\begin{align*}
    \int_0^1 V_r(r_0(t),0)u(t)\,dt
    &=
    -\frac{2B(\gamma_0)}{A(\gamma_0)}
    \int_0^1 r_0''(t)u(t)\,dt \\
    &=
    \frac{2B(\gamma_0)}{A(\gamma_0)}
    \int_0^1 r_0'(t)u'(t)\,dt.
\end{align*}
Therefore, the quadratic form in the $r$-direction can be written as
\begin{align*}
    Q_r(u_1, u_2) :=& \hat I''(\gamma_0)[(u_1,0),(u_2,0)]\\
    =&B(\gamma_0)\int_0^1 u_1'(t) u_2'(t) \,dt
    - \frac{4B(\gamma_0)}{A(\gamma_0)}
    \left( \int_0^1 r_0'(t)u_1'(t) \,dt \right)
    \left( \int_0^1 r_0'(t)u_2'(t) \,dt \right)\\
    &- \frac12 A(\gamma_0)
    \int_0^1 V_{rr}(r_0(t),0)u_1(t)u_2(t)\,dt.
\end{align*}

\begin{lem}\label{lem:r_morse_two}
    Suppose that
    \begin{align*}
        I_0(\gamma_0) := \int_0^1 V_{rr}(r_0(t),0)\,dt >0.
    \end{align*}
    Furthermore, define
    \begin{align*}
        &I_1(\gamma_0) := \int_0^1V_{rr}(r_0(t),0)r_0(t)\,dt,\\
        &I_2(\gamma_0) := \int_0^1 V_{rr}(r_0(t),0)r_0(t)^2\,dt,
    \end{align*}
    and assume that
    \begin{align*}
        I_0(\gamma_0) (6 B(\gamma_0) + I_2(\gamma_0)) - (I_1(\gamma_0))^2 > 0.
    \end{align*}
    Then the Morse index of the planar solution $\gamma_0$ in the $r$-direction is at least $2$.
\end{lem}

This assumption is satisfied, by the Cauchy--Schwarz inequality, if
\begin{equation*}
    V_{rr}(r_0(t))>0 \quad \forall t \in [0,1].
\end{equation*}
However, since this condition is too strong, we verify Lemma \ref{lem:r_morse_two} in applications. 

\begin{proof}
    First, for the constant function $1$, we have
    \begin{equation*}
        Q_r(1,1)
        = -\frac12 A(\gamma_0)\int_0^1 V_{rr}(r_0(t),0)\,dt < 0.
    \end{equation*}
    Next,
    \begin{align*}
        Q_r(r_0, r_0)
        &= -3 A(\gamma_0) B(\gamma_0)
        - \frac12 A(\gamma_0)
        \int_0^1 V_{rr}(r_0(t),0)r_0(t)^2\,dt,\\
        Q_r(1,r_0)
        &= -\frac12 A(\gamma_0)
        \int_0^1 V_{rr}(r_0(t),0)r_0(t)\,dt.
    \end{align*}
    Since $Q_r(1,1)<0$, $Q_r$ is negative definite on
    $\operatorname{span}\{1, r_0\}$ if and only if
    \begin{align*}
        \det \begin{pmatrix}
            Q_r(1,1) & Q_r(1, r_0)\\
            Q_r(1, r_0) & Q_r(r_0, r_0)
        \end{pmatrix}
       > 0.
    \end{align*}
    We compute
    \begin{align*}
        &\det \begin{pmatrix}
            Q_r(1,1) & Q_r(1, r_0)\\
            Q_r(1, r_0) & Q_r(r_0, r_0)
        \end{pmatrix}\\
        =& \frac14 A(\gamma_0)^2
        \left[
            I_0(\gamma_0) \left(6 B(\gamma_0) + I_2(\gamma_0)\right)
            - \left(I_1(\gamma_0) \right)^2
        \right] >0.
    \end{align*}
    This proves the lemma.
\end{proof}

\section{Application 1: the isosceles three-body problem}

\subsection{Existence result}

We examine an application to the spatial isosceles three-body problem, where two bodies of equal mass move symmetrically with respect to the $z$-axis, while the third body is constrained to move only along the $z$-axis.

The existence of nontrivial relative periodic orbits in the spatial isosceles three-body problem via variational methods was proved by Offin and Cabral\cite{OffinCabral09}, and
one of the authors \cite{Shibayama09}.
In both works, such orbits are obtained as minimizers of the Lagrangian action functional under a fixed-period constraint.
In contrast to previous works, our approach establishes the existence of (quasi-)periodic solutions with prescribed energy and angular momentum, rather than prescribed period.

Let the masses be
\begin{equation*}
    m_1=m_2=m,\qquad m_3=\alpha m\qquad (\alpha>0),
\end{equation*}
and let $q_i\in\mathbb R^3$ denote the position of the $i$-th body.
We write the positions as
\begin{equation*}
    q_1=\left(x,y,-\sqrt{\frac{\alpha}{\alpha+2}} z\right),\quad
    q_2=\left(-x,-y,-\sqrt{\frac{\alpha}{\alpha+2}} z\right),\quad
    q_3=\left(0,0,2 \sqrt{\frac{1}{\alpha(\alpha+2)}} z \right).
\end{equation*}
Then
\begin{equation*}
    m q_1+m q_2+\alpha m q_3=0,
\end{equation*}
so the center of mass is fixed at the origin.
The Lagrangian is given by
\begin{equation*}
    L = m(\dot x^2+\dot y^2+\dot z^2)
    + \frac{m^2}{2\sqrt{x^2 + y^2 }}
    +
    \frac{2\alpha m^2}{\sqrt{x^2 + y^2 + \frac{\alpha+2}{\alpha}z^2}}.
\end{equation*}

Using cylindrical coordinates $(r,\theta,z)$, we obtain
\begin{equation*}
    L = m(\dot r^2+r^2\dot\theta^2+\dot z^2)
    + \frac{m^2}{2r}
    + \frac{2\alpha m^2}{\sqrt{r^2+\frac{\alpha+2}{\alpha}z^2}}.
\end{equation*}
After a suitable scaling, this can be written as
\begin{align}
    L &= \frac12 ( \dot r^2 + r^2 \dot \theta^2 + \dot z^2)
    + \frac{1}{2r}
    + \frac{2 \alpha }{\sqrt{r^2 + \frac{\alpha + 2}{\alpha }z^2}},\label{eq:isosceles-lagrangian}\\
    H &= \frac12 \left( p_r^2 + \frac{p_\theta^2}{r^2} + p_z^2\right)
    - \frac{1}{2r}
    - \frac{2\alpha }{\sqrt{r^2 + \frac{\alpha + 2}{\alpha }z^2}}\label{eq:isosceles-hamiltonian}.
\end{align}
Setting $p_\theta = \omega$, we have
\begin{equation*}
    H = \frac 12 (p_r^2 + p_z^2) + V(r,z),
    \quad
    V(r,z) =
    \frac{\omega^2}{2r^2}
    - \frac{1}{2r}
    - \frac{2\alpha }{\sqrt{r^2 + \frac{\alpha + 2}{\alpha }z^2}}.
\end{equation*}

The only critical point of this potential $V(r,z)$ is
\begin{equation*}
    \left(\frac{2\omega^2}{1 + 4 \alpha},0\right).
\end{equation*}
Moreover, for
\begin{equation*}
    h \in \left(
    -\frac{(1 + 4 \alpha)^2}{8 \omega^2},
    - \frac{1}{8 \omega^2}
    \right),
\end{equation*}
the Hill region is compact and simply connected.
Since for $z>0$,
\begin{equation*}
    \partial_zV(r,z) = \frac{2(\alpha + 2)z}{(r^2 + \frac{\alpha + 2}{\alpha}z^2)^{3/2}}>0,
\end{equation*}
condition (V3) is satisfied.
Conditions (V1) and (V2) are clearly satisfied.
In the next subsection, we show that this solution is not planar.
Then, we can apply Theorem \ref{thm:main} to the problem, and obtain the following:

\begin{thm}
    For any $\omega >0$, $h \in \bigl(-\frac{(1 + 4 \alpha)^2}{8 \omega^2}, - \frac{1}{8 \omega^2}\bigr)$, and $\alpha >0$, the isosceles three-body problem with Hamiltonian \eqref{eq:isosceles-hamiltonian} admits a non-planar (quasi-)periodic solution with energy $h$ and angular momentum $\omega$.
\end{thm}

\subsection{Planar solutions}

In the isosceles three-body problem, planar solutions are Keplerian orbits.
Indeed,
\begin{align*}
    \bar r_0(\theta)
    =
    \frac{2 \omega^2}{(1 + 4 \alpha ) (1 + e \cos \theta)},
    \quad
    e = \sqrt{1 + \frac{8 h \omega^2}{(1 + 4 \alpha)^2}}
\end{align*}
is a solution.

\begin{lem}
    In the isosceles three-body problem, the Morse index of the planar solution $\gamma_0$ is at least $3$.
\end{lem}

\begin{proof}
    We first show that the Morse index in the $z$-direction is at least $1$.
    In the present case, $\theta_0=\pi$ and
    \begin{equation*}
        V_{zz}(r,0)=\frac{2(\alpha+2)}{r^3}.
    \end{equation*}
    Take
    \begin{equation*}
        \bar v(\theta):=\bar r_0(\theta)\sin\theta
        =
        \frac{2\omega^2}{1+4\alpha}
        \frac{\sin\theta}{1+e\cos\theta}.
    \end{equation*}
    Then
    \begin{align*}
        &\bar v(0)=\bar v(\pi)=0,\\
        &\bar v'(\theta)
        =
        \frac{2\omega^2}{1+4\alpha}
        \frac{\cos\theta+e}{(1+e\cos\theta)^2}.
    \end{align*}
    Therefore,
    \begin{align*}
        Q_z(\bar v)
        &=
        \int_0^\pi
        \left(
        \frac{\omega^2}{\bar r_0(\theta)^2}\bar v'(\theta)^2
        - \frac{2(\alpha+2)}{\bar r_0(\theta)}\bar v(\theta)^2
        \right)d\theta \\
        &=
        \omega^2 \int_0^\pi
        \left(
        \frac{(\cos\theta+e)^2}{(1+e\cos\theta)^2}
        - \frac{4(\alpha+2)}{1+4\alpha}
        \frac{\sin^2\theta}{1+e\cos\theta}
        \right)d\theta .
    \end{align*}
    Since
    \begin{align*}
        \int_0^\pi \left( \frac{(\cos\theta+e)^2}{(1+e\cos\theta)^2} - \frac{\sin^2\theta}{1+e\cos\theta}\right)\, d \theta  
        &= \int_0^\pi \frac{d}{d \theta} \left( \frac{\sin \theta (\cos \theta + e)}{1 + e \cos \theta} \right)\, d \theta = 0
    \end{align*}
    we have
    \begin{align*}
        \int_0^\pi\frac{(\cos\theta+e)^2}{(1+e\cos\theta)^2}d\theta
        =
        \int_0^\pi\frac{\sin^2\theta}{1+e\cos\theta} d\theta
        =
        \frac{\pi}{e^2}\bigl(1-\sqrt{1-e^2}\bigr).
    \end{align*}
    Therefore, we obtain
    \begin{equation*}
        Q_z(\bar v)
        =
        -\frac{7 \omega^2 \pi}{(1 + 4 \alpha) e^2}
        \bigl(1-\sqrt{1-e^2}\bigr) < 0.
    \end{equation*}
    Thus the Morse index in the $z$-direction is at least $1$.

    We next show that the Morse index in the $r$-direction is at least $2$.
    We use Lemma \ref{lem:r_morse_two}.
    Set
    \begin{equation*}
        \mu = \frac{1 + 4\alpha}{2},
        \quad
        \rho=\sqrt{1-e^2}.
    \end{equation*}
    Then, we obtain
    \begin{align*}
        V_{rr}(\bar r_0(\theta),0) &= \frac{3 \omega^2}{r_0(\theta)^4} - \frac{1 + 4 \alpha}{r_0(\theta)^3} = \frac{3 \omega}{r^4} - \frac{2\mu}{r^3} = \frac{\mu^4}{\omega^6}(1 + e \cos \theta)^3(1+3e \cos\theta) ,\\
        \sqrt{\lambda_0} &= \int_0^\pi \frac{ds}{d\theta} d\theta = \int_0^\pi \frac{\bar r_0(\theta)^2}{\omega}\, d \theta = \int_0^\pi \frac{\omega^3}{\mu^2 (1 + e \cos \theta)^2} \, d\theta = \frac{\omega^3 \pi}{\mu^2 \rho^3}\\
        dt &=  \frac{\rho^3}{\pi (1 + e \cos \theta)^2} d \theta 
    \end{align*}
    and 
    \begin{align*}
        I_0(\gamma_0)
        &= \int_0^1 V_{rr}(r_0(t),0)\,dt 
        = \int_0^\pi V_{rr}(\bar r_0(\theta),0) \frac{\rho^3}{\pi (1 + e \cos \theta)^2} d \theta  
        = \frac{\mu^4\rho^3}{2\omega^6}(5-3\rho^2),\\
        I_1(\gamma_0)
        &= \int_0^1 V_{rr}(r_0(t),0)r_0(t)\,dt = \int_0^\pi V_{rr}(\bar r_0(\theta),0) \frac{\omega^2 \rho^3}{\mu \pi (1 + e \cos \theta)^3} d \theta  
        =\frac{\mu^3\rho^3}{\omega^4},\\
        I_2(\gamma_0)
        &= \int_0^1 V_{rr}(r_0(t),0)r_0(t)^2\,dt = \int_0^\pi V_{rr}(\bar r_0(\theta),0) \frac{\omega^4 \rho^3}{\mu^2 \pi (1 + e \cos \theta)^4} d \theta
        =\frac{\mu^2\rho^3}{\omega^2}\left(3-\frac{2}{\rho}\right),\\
        B(\gamma_0)
        &= \int_0^1 (h - V(r_0(t),0))dt
        = - \frac{\mu^2 \rho^2}{2\omega^2}- \int_0^\pi  \frac{\omega^2}{2\bar r_0(\theta)^2} - \frac{\mu}{\bar r_0(\theta)} \,d\theta
        = \frac{\mu^2\rho^2(1-\rho)}{2\omega^2}.
    \end{align*}
    In particular,
    \begin{equation*}
        I_0(\gamma_0)>0.
    \end{equation*}
    Furthermore,
    \begin{align*}
        &I_0(\gamma_0)\bigl(6B(\gamma_0)+I_2(\gamma_0)\bigr)
        -I_1(\gamma_0)^2\\
        &\qquad =
        \frac{\mu^6\rho^5}{2\omega^8}(1-\rho)(5+3\rho)>0.
    \end{align*}
    Therefore, by the lemma for the $r$-direction, the Morse index in the $r$-direction is at least $2$.

    Combining the two estimates, the total Morse index of a planar solution is at least $3$.
\end{proof}

\section{Application 2: The spatial anisotropic Kepler problem}
A Hamiltonian system with Hamiltonian
\begin{equation}\label{eq:anisotropic-kepler}
    H = \frac 12 \left( p_r^2 +  \frac{p_\theta^2}{r^2} + p_z^2 \right)
    - \frac{1}{\sqrt{r^2 + (1 + \beta) z^2}},\quad (\beta >-1) 
\end{equation}
is called the spatial anisotropic Kepler problem.

The anisotropic Kepler problem is introduced by Gutzwiler to model an electron in semiconductors with donor impurities (\cite{Gutzwiller73}).

Periodic orbits in the problem have recently been studied by various methods.
Guirao, Llibre, and Vera \cite{GuiraoLlibreVera13} studied the problem as a special case of the perturbed spatial Kepler problem and showed the existence of two $2\pi$-periodic unstable orbits with zero third angular momentum on each negative energy surface by averaging theorem. 
Li and Liu \cite{LiLiu21} studied the problem with anisotropic perturbations.
Hu, Ou and Qiao \cite{HuOuQiao26}, using contact geometry, showed that after reducing the system modulo rotational symmetry around the $z$-axis, there exist infinitely many periodic orbits on any fixed compact and regular reduced energy surface for $\beta \in (-1, 0]$.
One of the authors \cite{Sakaguchi} show the existence of periodic solutions for $\beta > 0$ using variational minimizing methods based on the action functional, in which the period is fixed.

We fix the angular momentum at $\omega$.
For
\begin{align*}
    V(r,z) &= \frac{\omega^2}{2r^2}
    - \frac{1}{\sqrt{r^2 + (1 + \beta)z^2}},
\end{align*}
if
\begin{equation*}
    h \in \left(-\frac{1}{2 \omega^2}, 0\right),
\end{equation*}
then the Hill region is compact and simply connected, and condition (V3) is also satisfied since $\partial_z V(r,z)>0$ if $z > 0$.
Moreover, the only critical point of $V$ is $(\omega^2,0)$.

Here again, planar solutions are Keplerian orbits.
As in the isosceles three-body problem, the Morse index in the $r$-direction is larger than or equal to $2$.
We now compute the Morse index in the $z$-direction.
The planar solution can be written as
\begin{equation*}
    \bar r_0(\theta) = \frac{\omega^2}{1 + e \cos \theta},
    \quad 0 < e < 1.
\end{equation*}
Since
\begin{equation*}
    V_{zz}(r,0) = \frac{1+\beta}{r^3},
\end{equation*}
we have
\begin{align*}
    Q_z(\bar v)
    &=
    \int_0^\pi
    \left(
        \frac{\omega^2}{\bar r_0(\theta)^2}\bar v'(\theta)^2
        -
        \frac{1 + \beta}{\bar r_0(\theta)}\bar v(\theta)^2
    \right)d\theta .
\end{align*}
Taking
\begin{equation*}
    \bar v(\theta) = \bar r_0(\theta)\sin\theta,
\end{equation*}
we obtain
\begin{align*}
    Q_z(\bar v)
    &=
    \omega^2 \int_0^\pi
    \left(
        \frac{(e + \cos \theta)^2}{(1 + e\cos \theta)^2}
        -
        \frac{(1 + \beta) \sin ^2 \theta}{1 + e \cos \theta}
    \right)d \theta\\
    &=
    -\frac{\beta \omega^2\pi}{e^2}
    \left(1 - \sqrt{1 - e^2}\right).
\end{align*}
Therefore, if $\beta>0$, then the Morse index in the $z$-direction is at least $1$.
It follows that, for $\beta>0$, the planar solution cannot coincide with the relative periodic orbit obtained by the minimax method.
Therefore, we obtain the following:

\begin{thm}
    For any $\omega >0$, $h \in (-\frac{1}{2\omega^2},0)$, and $\beta>0$, the spatial anisotropic Kepler problem with Hamiltonian \eqref{eq:anisotropic-kepler}  admits a non-planar (quasi-)periodic solution with energy $h$ and angular momentum $\omega$.    
\end{thm}

\section*{Acknowledgments}
M. S. was supported by JSPS KAKENHI Grant Number JP23K25778.

\section*{Data availability statement}
No datasets were generated or analyzed during the current study.

\section*{Conflict of interest}
The authors declare that they have no conflict of interest.

\bibliography{list}
\bibliographystyle{abbrv}

\end{document}